\documentclass[a4paper,12pt]{article}
\usepackage{amsmath,amssymb,bm}
\usepackage[dvipdfmx]{graphicx,xcolor}
\usepackage{enumerate}
\usepackage{tikz}
\usetikzlibrary{intersections,calc,arrows}
\usepackage{mathrsfs}
\usepackage{amsthm}

\title{Clarke Differentials and the Envelope Theorem in Dynamic Programming}
\author{Yuhki Hosoya\thanks{E-mail: hosoya(at)tamacc.chuo-u.ac.jp,\ ORCID ID:0000-0002-8581-4518}\\
Faculty of Economics, Chuo University}
\makeatletter
\renewenvironment{proof}[1][\proofname]{\par
  \pushQED{\qed}%
  \normalfont \topsep6\p@\@plus6\p@\relax
  \trivlist
  \item\relax
  {\bfseries
  #1\@addpunct{.}}\hspace\labelsep\ignorespaces
}{%
  \popQED\endtrivlist\@endpefalse
}
\makeatother

\theoremstyle{definition}
\newtheorem{prop}{Proposition}
\newtheorem{thm}{Theorem}
\newtheorem{lem}{Lemma}
\newtheorem{cor}{Corollary}

\renewcommand{\proofname}{Proof}

\begin{document}
\maketitle

\begin{abstract}
In this paper, we consider a deterministic dynamic programming model, and derive the envelope theorem using the Clarke differential. Compared with previous research, we do not require differentiability, convexity, or boundedness.

\vspace{12pt}

\vspace{12pt}
\end{abstract}

\section{Introduction}
Although the term `envelope theorem' covers many different results in various fields of applied mathematics, one of the simplest is the following. First, consider the one-dimensional parametrized optimization problem of $f(x,p)$ with respect to $x$, and let $x(p)$ be the solution function. Define the value function $V(p)=f(x(p),p)$. Suppose that $f$ is differentiable at $(x(p),p)$, and that $x$ is also differentiable at $p$. Then,
\[V'(p)=\frac{\partial f}{\partial x}(x(p),p)x'(p)+\frac{\partial f}{\partial p}(x(p),p).\]
However, by the first-order condition, we have that $\frac{\partial f}{\partial x}(x(p),p)=0$, which leads to the following result:
\[V'(p)=\frac{\partial f}{\partial p}(x(p),p).\]
This is the envelope theorem. Although this result can be extended to many optimization problems, its proof is usually heavily dependent on the differentiability of $f$ and $x$. Hence, when $f$ is not differentiable, this result may be inapplicable.

On the other hand, in dynamic programming models, the value function $\bar{V}(x)$ satisfies the following {\bf Bellman equation}:
\[\bar{V}(x)=\max\{w(x,y)+\delta \bar{V}(y)|y\in \Gamma(x)\}.\]
If the envelope theorem can be applied, then we have that
\[\bar{V}'(x)=\frac{\partial w}{\partial x}(x,y^*(x)),\]
where $y^*(x)$ is an optimizer for the maximization problem on the right-hand side of the Bellman equation. However, there is a problem. To prove this result, the objective function $w(x,y)+\delta \bar{V}(y)$ must be differentiable. Because $w$ is the exogenously determined objective function, we can assume the differentiability of $w$. However, $\bar{V}$ is a function derived endogenously from the model, and thus it is difficult to exogenously assume its differentiability. Furthermore, even if $\bar{V}$ is differentiable, the above proof of the envelope theorem requires the differentiability of $y^*$. If this problem is not resolved, the envelope theorem cannot be applied. 

Theorem 1 of Benveniste and Scheinkman \cite{BS} is the first result to overcome the above difficulty. They introduced some convex structure in the model and showed the envelope theorem without assuming the differentiability of $\bar{V}$ and $y^*$ using the techniques of convex analysis. This result is referred to in many macroeconomic textbooks as the `Benveniste--Scheinkman envelope theorem'.

Although this result is good, several of the assumptions used are somewhat strong. Hence, several relaxed results of this `envelope theorem' have been proved. One such result is Theorem 5.2 of Mordukhovich and Sagara \cite{MS}. They used the Clarke differential introduced in Clarke \cite{C83} instead of convex analysis. Their result can be applied to stochastic dynamic programming models in which the variable is in some separable Banach space, and thus the scope of applicable models became much wider.

However, a problem remains. Although Mordukhovich--Sagara included stochastic models in their application, stochastic dynamic programming models include several technical problems that typically do not arise in deterministic models. For example, the value function may not be measurable in some models. To address this issue, they have made some strong assumptions. For example, the domain of the objective function must be the entire space, and the value of this function must be bounded. These assumptions are not valid in many economic applications.\footnote{As will be discussed in detail in subsection 3.2, there exists a one-parameter family of functions frequently appearing in economic theory named CRRA, which includes the natural logarithm function. Since the natural logarithm function neither has the entire real line as its domain nor is bounded, it is highly incompatible with Mordukhovich--Sagara's assumption.}

In this paper, we restrict our discussion to deterministic models and reproduce the results of Mordukhovich and Sagara \cite{MS} under weaker assumptions. The main result is Theorem \ref{thm3}, which proves the envelope theorem under very weak assumptions that are satisfied by many economic models.

This paper makes extensive use of the Clarke differential. To make the paper self-contained, we introduce several fundamental results related to the Clarke differential with proofs. Of course, most of these results have already been shown by Clarke \cite{C83}. For the sake of readability, the proofs of these results are included in the appendix.

The structure of this paper is as follows. In sect.2, we define several terms, including the Clarke differential, and discuss their properties. In sect.3, we discuss the main results. A brief conclusion is presented in sect.4.

\section{Definitions and Preliminary Results}
Throughout this section, the symbols $X,Y$ basically denote some Banach spaces and $X'$ denotes the dual space of $X$. For $x\in X$ and $r>0$, define $B_r(x)=\{y\in X|\|y-x\|<r\}$. This set $B_r(x)$ is called the {\bf open ball centered at $x$ with radius $r$.} Similarly, define $\bar{B}_r(x)=\{y\in X|\|y-x\|\le r\}$. This set $\bar{B}_r(x)$ is called the {\bf closed ball centered at $x$ with radius $r$.}

\subsection{Locally Lipschitz Functions}
Let $U\subset X$ be a nonempty set and suppose that $f:U\to Y$ is given. This function $f$ is said to be {\bf $L$-Lipschitz} if the following inequality
\begin{equation}\label{L}
\|f(x)-f(y)\|\le L\|x-y\|
\end{equation}
holds for any $x,y\in U$. A function $f:U\to Y$ is said to be {\bf Lipschitz} if and only if there exists $L>0$ such that $f$ is $L$-Lipschitz. We call such an $L$ a {\bf Lipschitz constant} for $f$. If $C\subset U$ is also nonempty, then $f$ is said to be {\bf Lipschitz on $C$} if and only if the restriction of $f$ on $C$ is Lipschitz.

A natural extension to this concept is as follows. $f:U\to Y$ is said to be {\bf locally Lipschitz} if for each $x\in U$, there exists $r>0$ such that $f$ is Lipschitz on $B_r(x)\cap U$. We need the following result.

\begin{prop}\label{prop1}
\begin{enumerate}[1)]
\item Every locally Lipschitz function $f:U\to Y$ is Lipschitz on any compact subset $C\subset U$.

\item If $X=\mathbb{R}^n$ and $U\subset \mathbb{R}^n$ is open, then $f:U\to Y$ is locally Lipschitz if and only if $f$ is Lipschitz on any compact subset $C\subset U$.
\end{enumerate}
\end{prop}

Next, recall the definition of several notions for derivatives. A function $T:X\to Y$ is said to be {\bf linear} if and only if for each $x,y\in X$ and $a,b\in \mathbb{R}$,
\[T(ax+by)=aT(x)+bT(y).\]
Note that, when $T$ is linear, it is customary to write $Tx$ instead of $T(x)$. If $Y$ is $\mathbb{R}$, then $Tx$ is also written as $\langle T,x\rangle$. It is well-known that if $T:X\to Y$ is linear, then $T$ is continuous if and only if $T$ is Lipschitz, and if we define
\[\|T\|=\inf\{L>0|L\mbox{ is a Lipschitz constant of }T\}=\sup\{\|Tx\||\|x\|\le 1\},\]
then $\|\cdot\|$ satisfies all axioms of the norm, and the set of all continuous linear function $L(X,Y)$ becomes a Banach space with this norm $\|\cdot\|$. This norm is called the {\bf operator norm} of $T$. In particular, $L(X,\mathbb{R})$ is called the {\bf dual space} of $X$, denoted by $X'$.

Suppose that $U\subset X$ is a nonempty open set, and that $f:U\to Y$ and $x\in U$ are given. Suppose that there exists $T\in L(X,Y)$ such that, for each $v\in X$,
\[\lim_{t\downarrow 0}\frac{f(x+tv)-f(x)}{t}=Tv.\]
Then, $T$ is called the {\bf G\^ateaux derivative} of $f$ at $x$, denoted by $D_Gf(x)$. Moreover, if the above convergence is uniform in $v$ on every compact set $C$, then $T$ is called the {\bf Hadamard derivative} of $f$ at $x$, denoted by $D_Hf(x)$. Finally, if the above convergence is uniform in $v$ on every bounded set $B$, then $T$ is called the {\bf Fr\'echet derivative} of $f$ at $x$, denoted by $Df(x)$.

If the mapping $x\mapsto D_Gf(x)$ can be defined and is continuous, then we can show that $D_Gf(x)=Df(x)$.\footnote{See Proposition \ref{prop6}.} In this case, $f$ is said to be {\bf continuously differentiable}, or $C^1$. Using the mean value theorem, we can easily show that every $C^1$ function is locally Lipschitz.

If $X=\mathbb{R}^n, Y=\mathbb{R}^m$, then the `almost everywhere' converse relationship holds. That is, if $U\subset X$ is nonempty and open, then every locally Lipschitz function $f:U\to Y$ is Fr\'echet differentiable almost everywhere (Rademacher's theorem). For a detailed argument, see Heinonen \cite{H}.

A set $C\subset X$ is said to be {\bf convex} if and only if $x,y\in C$ and $t\in [0,1]$ implies $(1-t)x+ty\in C$. Suppose that $U\subset X$ is a nonempty and convex subset. A real-valued function $f:U\to \mathbb{R}$ is said to be {\bf convex} if and only if the epigraph $E=\{(x,a)\in U\times \mathbb{R}|f(x)\le a\}$ is convex, and, $f$ is said to be {\bf concave} if and only if $-f$ is convex. If $X=\mathbb{R}^n$, then it is well-known that any convex (or, concave) function defined on an open set is locally Lipschitz. See Theorem 10.4 of Rockafellar \cite{R}.

\subsection{Clarke Derivatives}
We start from the definition of generalized directional derivatives. Suppose that $U\subset X$ is a nonempty open subset, a real-valued function $f:U\to \mathbb{R}$ is given, and $x\in U, v\in X$. Define
\[f^{\circ}(x;v)=\limsup_{y\to x,\ t\downarrow 0}\frac{f(y+tv)-f(y)}{t}.\]
We call this value $f^{\circ}(x;v)$ the {\bf generalized directional derivative of $f$ at $x$ with direction $v$}. The following result is useful.\footnote{A function $\varphi:X\to \mathbb{R}$ is {\bf sublinear} if 1) for any $a\ge 0$, $\varphi(av)=a\varphi(v)$, and 2) $\varphi(v+w)\le \varphi(v)+\varphi(w)$.}

\begin{prop}\label{prop2}
Suppose that $U\subset X$ is a nonempty open set and $f:U\to \mathbb{R}$ is locally Lipschitz. For a fixed $x\in U$, $f^{\circ}(x;v)$ is well-defined and a real number. If $f$ is $L$-Lipschitz around $x$, then $|f^{\circ}(x;v)|\le L\|v\|$. Furthermore, the mapping $v\mapsto f^{\circ}(x;v)$ is sublinear.
\end{prop}

Now, suppose that $U\subset X$ is a nonempty open set, and $f:U\to \mathbb{R}$ is locally Lipschitz. Define a set-valued function $\partial^{\circ}f:X\twoheadrightarrow X'$ as
\[\partial^{\circ}f(x)=\{p\in X'|\forall v\in X,\ \langle p,v\rangle \le f^{\circ}(x;v)\}.\]
We call this function $\partial^{\circ}f$ the {\bf Clarke differential} of $f$.

Using Proposition \ref{prop2}, we can easily check that $\partial^{\circ}f(x)$ is nonempty, convex, and weak* compact. First, because $v\mapsto f^{\circ}(x;v)$ is sublinear, the Hahn-Banach theorem can be applied, and thus there exists a linear functional $p:X\to \mathbb{R}$ such that $\langle p,v\rangle\le f^{\circ}(x;v)$ for any $v$. Now, because $f$ is assumed to be locally Lipschitz, there exists $L>0$ such that $f$ is $L$-Lipschitz around $x$. By Proposition \ref{prop2}, $|f^{\circ}(x;v)|\le L\|v\|$, and thus $p$ is $L$-Lipschitz. This implies that $p\in X'$, and thus $\partial^{\circ}f(x)$ is nonempty. It is clear that $\partial^{\circ}f(x)$ is convex and weak* closed. Because $p\in \partial^{\circ}f(x)$ implies that $p$ is $L$-Lipschitz, $\partial^{\circ}f(x)$ is bounded. By Alaoglu's theorem, $\partial^{\circ}f(x)$ is weak* compact.

The following proposition is very useful.

\begin{prop}\label{prop3}
Suppose that $U\subset X$ is a nonempty open subset of a Banach space $X$, and $f:U\to \mathbb{R}$ is locally Lipschitz. Then, the following equality holds:
\[-\partial^{\circ}f(x)=\partial^{\circ}(-f)(x).\]
\end{prop}

Using this proposition, we can show the following fundamental result for the optimization theory.

\begin{cor}\label{cor1}
Suppose that $U$ is a nonempty open subset of a Banach space $X$, and $f:U\to \mathbb{R}$ is locally Lipschitz. Choose $x^*\in U$. If $x^*$ is either a local maximum point or a local minimum point, then $0\in \partial^{\circ}f(x^*)$.
\end{cor}

\begin{proof}
First, suppose that $x^*$ is a local minimum point. Then, $f^{\circ}(x^*;v)\ge 0$ by definition, and thus, $0\in \partial^{\circ}f(x)$. Next, suppose that $x^*$ is a local maximum point. Then, $x^*$ is a local minimum point of $-f$. Therefore, $0\in \partial^{\circ}(-f)(x^*)$. By Proposition \ref{prop3},
\[0=-0\in -\partial^{\circ}(-f)(x^*)=\partial^{\circ}f(x^*).\]
This completes the proof.
\end{proof}

\vspace{12pt}
\noindent
{\bf NOTE}. In several textbooks, $\partial^{\circ}f(x)$ is called the Clarke {\bf subdifferential}. Actually, there is another notion of generalized differentials that is called the Clarke {\bf superdifferential}, which was introduced in Sagara \cite{S}. Define
\[f^*(x;v)=\liminf_{y\to x,\ t\downarrow 0}\frac{f(y+tv)-f(y)}{t}\]
and
\[\partial^*f(x)=\{p\in X'|\langle p,v\rangle\ge f^*(x;v)\mbox{ for all }v\in X\}.\]
This $\partial^*f(x)$ is called the Clarke superdifferential.

However, this `Clarke superdifferential' is actually the same as the `Clarke subdifferential'. That is, $\partial^*f(x)=\partial^{\circ}f(x)$. To show this, we first note the following equation:
\begin{align*}
-f^*(x;v)=&~-\liminf_{y\to x,\ t\downarrow 0}\frac{f(y+tv)-f(y)}{t}\\
=&~\limsup_{y\to x,\ t\downarrow 0}\frac{(-f)(y+tv)-(-f)(y)}{t}=(-f)^{\circ}(x;v).
\end{align*}
Therefore, for any $v\in X$,
\begin{align*}
\langle p,v\rangle \ge f^*(x;v)\Leftrightarrow&~\langle p,v\rangle\ge -(-f)^{\circ}(x;v)\\
\Leftrightarrow&~\langle -p,v\rangle\le (-f)^{\circ}(x;v).
\end{align*}
In particular, 
\[p\in \partial^*f(x)\Leftrightarrow -p\in \partial^{\circ}(-f)(x)=-\partial^{\circ}f(x),\]
which implies that $\partial^*f(x)=\partial^{\circ}f(x)$. Therefore, the Clarke `subdifferential' coincides with the `superdifferential', and hence we simply call $\partial^{\circ}f(x)$ the Clarke `differential'.

We additionally note the relationship between the usual `subdifferential' and `Clarke differential'. Suppose that $U$ is an open and convex subset of a Banach space $X$. By definition, $f:U\to \mathbb{R}$ is said to be {\bf convex} if and only if for any $x,y\in U$ and $t\in [0,1]$,
\[f((1-t)x+ty)\le (1-t)f(x)+tf(y).\]
If $f$ is convex, then the set
\[\partial f(x)=\{p\in X'|\langle p,v\rangle \le f(x+v)-f(x)\mbox{ for all }v\in X\}\]
is called the {\bf subdifferential} for the function $f$ at $x$. Then, the following proposition holds.\footnote{Note that, if $X=\mathbb{R}^n$, then any convex function $f:U\to \mathbb{R}$ is locally Lipschitz.}

\begin{prop}\label{prop4}
Suppose that $U$ is a nonempty, open, and convex subset of a Banach space $X$, and $f:U\to \mathbb{R}$ is convex and locally Lipschitz. Then, $\partial f(x)=\partial^{\circ}f(x)$.
\end{prop}

\subsection{Strong Differentiability}
For a subdifferential of the concave function $f:I\to \mathbb{R}$ defined on an open interval $I$, $\partial f(x)$ is a singleton if and only if $f$ is differentiable at $x$. For the Clarke differential, however, this result does not hold. For example, consider the following function
\[f(x)=\begin{cases}
x^2\sin(1/x) & \mbox{if }x\neq 0,\\
0 & \mbox{if }x=0.
\end{cases}\]
If $x\neq 0$, then
\[f'(x)=2x\sin(1/x)-\cos(1/x),\]
and thus, $|f'(x)|\le 2x+1$. Next,
\[f'(0)=\lim_{h\to 0}\frac{f(h)-f(0)}{h}=\lim_{h\to 0}h\sin(1/h)=0.\]
By the mean value theorem, we have that $f$ is $(2M+1)$-Lipschitz on $[-M,M]$, and thus it is locally Lipschitz. However, we can show that
\[\partial^{\circ}f(0)=[-1,1],\]
and thus, $\partial^{\circ}f(0)\neq \{f'(0)\}$ in this case.

To solve this problem, we introduce the notion of {\bf strong differentiability}. Suppose that $U$ is a nonempty open set of a Banach space $X$, and $f:U\to \mathbb{R}$ is locally Lipschitz. Suppose that there exists $p\in X'$ such that
\[\lim_{y\to x,\ t\downarrow 0}\frac{f(y+tv)-f(y)}{t}=\langle p,v\rangle\]
for all $v\in X$. Then, $f$ is said to be {\bf strongly differentiable} at $x$, and $p$ is called the {\bf strong derivative} of $f$ at $x$, denoted by $D_sf(x)$.

We additionally introduce the {\bf regularity} of $f$ at $x$. Suppose that $U$ is a nonempty open subset of a Banach space $X$, and $f:U\to \mathbb{R}$ is locally Lipschitz. We define the classical directional derivative $f'(x;v)$ as
\[f'(x;v)=\lim_{t\downarrow 0}\frac{f(x+tv)-f(x)}{t}.\]
The function $f$ is said to be regular at $x$ if $f^{\circ}(x;v)=f'(x;v)$ for any $v\in X$.

The following proposition clarifies the relationship between Clarke derivative, usual derivatives, the strong derivative, and the regularity.

\begin{prop}\label{prop5}
Suppose that $U$ is a nonempty open subset of a Banach space $X$, $f:U\to \mathbb{R}$ is locally Lipschitz, and $x\in U$. Then, the following results hold.
\begin{enumerate}[(i)]
\item If $f$ is G\^ateaux differentiable at $x$, then $D_Gf(x)\in \partial^{\circ}f(x)$.

\item $\partial^{\circ}f(x)$ is a singleton if and only if $f$ is strongly differentiable. In this case, $\partial^{\circ}f(x)=\{D_sf(x)\}$.

\item If $f$ is strongly differentiable at $x$, then it is Hadamard differentiable at $x$, and $D_sf(x)=D_Hf(x)$.\footnote{If $X=\mathbb{R}^n$, then $D_Hf(x)=Df(x)$, and thus, $D_sf(x)=Df(x)$.}

\item If $f$ is strongly differentiable at $x$, then $f$ is regular at $x$.
\end{enumerate}
\end{prop}

The above different notions on differentiability can be simplified by assuming continuous differentiability. That is, the following proposition holds.

\begin{prop}\label{prop6}
Suppose that $U$ is a nonempty open subset of a Banach space $X$, and $f:U\to \mathbb{R}$ is G\^ateaux differentiable at any point. Moreover, suppose that $D_Gf:U\to X'$ is continuous at $x\in U$. Then, $f$ is both Fr\'echet differentiable and strongly differentiable at $x$, and $Df(x)=D_sf(x)=D_Gf(x)$.
\end{prop}

Finally, we present a proposition that ensures the regularity of $f$ at $x$.

\begin{prop}\label{propx}
Suppose that $U$ is a nonempty open and convex subset of a Banach space $X$, and $f:U\to \mathbb{R}$ is locally Lipschitz and either concave or convex. Then, $f$ is regular at any point.
\end{prop}

\section{Main Result}
\subsection{The Model}
We consider the following model:
\begin{align}
\max~~~~~&~\sum_{t=0}^{\infty}\delta^tw(x_t,x_{t+1})\nonumber \\
\mbox{subject to }&~x_{t+1}\in \Gamma(x_t),\label{RF}\\
&~x_0=\bar{x}\in W,\nonumber
\end{align}
where $W$ is a subset of a Banach space $X$, $\bar{W}$ is the closure of $W$, the function $w:T\to \mathbb{R}\cup\{-\infty\}$ and the set-valued function $\Gamma:\bar{W}\twoheadrightarrow \bar{W}$ are given (where $T$ denotes the graph of $\Gamma$), and $\delta\in ]0,1[$.

This type of model is referred to as a reduced form model, and many economic dynamic models can be transformed into this form in some way. Therefore, the reduced form model is more general than each specified model, which means that the conclusions obtained from it are more powerful. However, to examine the natural assumptions for $w$, $\Gamma$, and $W$, we must consider what $w, \Gamma, W$ represent in economic models. Hence, we present a specific example of representative dynamic economic model and examine the properties of $w$, $\Gamma$, and $W$ in this model.

We introduce the discretized version of the Ramsey--Cass--Koopmans type capital accumulation model (hereafter, RCK model). This model is represented by the following optimization problem. 
\begin{align}
\max~~~~~&~\sum_{t=0}^{\infty}\delta^tu(c_t)\nonumber \\
\mbox{subject to }&~k_t\ge 0,\ c_t\ge 0,\nonumber \\
&~c_t+i_t=f(k_t),\label{RCK}\\
&~k_{t+1}=(1-d)k_t+i_t,\nonumber \\
&~k_0=\bar{k}>0.\nonumber
\end{align}
Here, $c_t$ represents consumption, $i_t$ represents investment, and $k_t$ represents the capital stock. The proportion $d$ represents the capital depreciation rate. At the end of $t$-th period, some of the capital stock is broken, but some is newly created. $dk_t$ represents the amount of the capital stock that is broken, and $i_t$ represents the amount of the capital stock that is newly created. Hence, $k_{t+1}$ is equal to $(1-d)k_t+i_t$. The assumptions $k_t\ge 0,c_t\ge 0$ are natural. Although $i_t\ge 0$ is also natural, we ignore this assumption for simplicity. The function $f$ is the production function, and the gross product $f(k_t)$ is distributed by consumption $c_t$ and investment $i_t$. The function $u$ is called the {\bf utility function}, which evaluates the wellness of the consumption level. Hence, the RCK model represents the optimization problem for the plan of capital accumulation using the evaluation function $u$.

In this model, $i_t=f(k_t)-c_t$, and thus it can be omitted. Hence,
\[c_t=f(k_t)+(1-d)k_t-k_{t+1},\]
and thus $c_t$ can also be omitted. Therefore, this model can be transformed into the following model.
\begin{align*}
\max~~~~~&~\sum_{t=0}^{\infty}\delta^tu(f(k_t)+(1-d)k_t-k_{t+1})\\
\mbox{subject to. }&~k_t\ge 0,\ f(k_t)+(1-d)k_t-k_{t+1}\ge 0,\\
&~k_0=\bar{k}>0.
\end{align*}
Thus, if we define
\[w(x,y)=u(f(x)+(1-d)x-y),\]
\[\Gamma(x)=\{y\in \mathbb{R}|0\le y\le f(x)+(1-d)x\},\]
\[W=\mathbb{R}_{++}\equiv \{x\in \mathbb{R}|x>0\},\ \bar{W}=\mathbb{R}_+\equiv \{x\in \mathbb{R}|x\ge 0\},\]
then this model coincides with (\ref{RF}).

In (\ref{RCK}), the typical examples of $u,f$ are as follows:
\[u(c)=u_{\theta}(c)\equiv \begin{cases}
\log c & \mbox{if }\theta=1,\\
\frac{c^{1-\theta}-1}{1-\theta} & \mbox{if }\theta\neq 1,
\end{cases}\]
\[f(k)=k^a\mbox{ or }f(k)=Ak,\]
where $\theta>0$, $0<a<1$, and $A>0$. $u_{\theta}$ is the solution to the following differential equation:
\[-\frac{xu''(x)}{u'(x)}=\theta,\ u'(1)=1,\ u(1)=0.\]
Because the left-hand side is called the {\bf relative risk aversion} of $u$, this function is called the {\bf constant relative risk aversion} (CRRA) function. The function $f(k)=k^a$ is called the {\bf Cobb--Douglas technology}, and $f(k)=Ak$ is called the {\bf AK technology}. These are very common in macroeconomic dynamics. Hence, we must make suitable assumptions for $w,\Gamma,W$ that do not prohibit these examples. This means that $w$ is not necessarily bounded, and the value may be $-\infty$ (because $u(0)$ may be $-\infty$).

\subsection{The Result}
We consider the following functional equation on the unknown function $V$:
\begin{equation}\label{B}
V(x)=\max\{w(x,y)+\delta V(y)|y\in \Gamma(x)\}.
\end{equation}
This equation is called the {\bf Bellman equation}.

Let
\[\bar{V}(x)=\sup\left\{\left.\sum_{t=0}^{\infty}\delta^tw(x_t,x_{t+1})\right|x_{t+1}\in \Gamma(x_t)\mbox{ for all }t\ge 0\right\}.\]
This function $\bar{V}:W\to \mathbb{R}\cup\{-\infty,+\infty\}$ is called the {\bf value function} of (\ref{RF}). In many cases, $\bar{V}$ solves the Bellman equation (\ref{B}). If so, define
\[G(x)=\arg\max\{w(x,y)+\delta \bar{V}(y)|y\in \Gamma(x)\}.\]
The multi-valued function $G$ is called the {\bf policy function} of (\ref{RF}). It is well-known that $(x_t)$ is a solution to (\ref{RF}) if and only if $x_{t+1}\in G(x_t)$ for all $t\ge 0$. See ch.4 of Stokey and Lucas \cite{SL} for detailed arguments.

The following theorem was proved by Benveniste and Scheinkman \cite{BS}, and is so-called the Benveniste--Scheinkman envelope theorem.

\begin{thm}\label{thm1}
Suppose that $X=\mathbb{R}^n$. Let $T$ be the graph of $\Gamma$ and $\mbox{int.}T$ be the interior of $T$. Suppose that the following requirements hold.
\begin{enumerate}[i)]
\item $T$ is convex and $\mbox{int.}T$ is nonempty.

\item $w$ is continuous and concave on $T$, and real-valued and $C^1$ on $\mbox{int.}T$.

\item For any $\bar{x}\in W$, there exists $\bar{y}\in G(\bar{x})\cap \mbox{int.}T$.
\end{enumerate}
Then, $\bar{V}$ is $C^1$ on $W$, and for $\bar{y}\in G(\bar{x})\cap \mbox{int.}T$,
\[\bar{V}'(\bar{x})=\frac{\partial w}{\partial x}(\bar{x},\bar{y}).\]
\end{thm}

However, the requirements in Theorem \ref{thm1} are sometimes strong, and thus it is desirable to relax the assumptions. In this respect, Mordukhovich and Sagara \cite{MS} derived the following theorem. First, fix $y\in \bar{W}$ and define
\[w_y(x)=w(x,y),\ \partial^{\circ}_xw(x,y)=\partial^{\circ}w_y(x).\]

\begin{thm}\label{thm2}
Let $X$ be a separable Banach space, $W=X$, $w:W^2\to \mathbb{R}$ be continuous and bounded, and $\Gamma:W\twoheadrightarrow W$ be nonempty-valued, compact-valued, and continuous. Choose any $\bar{x}\in W$, and suppose that the following requirements hold.
\begin{enumerate}[i)]
\item There exists an open neighborhood $W'$ of $\bar{x}$ such that if $x,x'\in W'$, then $G(x)\neq \emptyset$ and $G(x)\subset \Gamma(x')$.

\item $w$ is Lipschitz around $(x,y)$ for all $x\in W'$ and $y\in G(x)$.

\item For every $\bar{y}\in G(\bar{x})$, $w$ is regular at $(\bar{x},\bar{y})$.\footnote{Mordukhovich and Sagara \cite{MS} stated that this theorem holds if i), ii) holds and $w$ is regular at $(\bar{x},\bar{y})$ for {\bf some} $\bar{y}\in G(\bar{x})$. However, we believe that under this condition, their proof of this theorem includes a gap. To be specific, since we cannot choose $\bar{y}$ appearing in the fifth paragraph of the proof of Theorem \ref{thm3}, the `for some' assumption cannot ensure that $w$ is regular at $(\bar{x},\bar{y})$ for this $\bar{y}$.}
\end{enumerate}
Then, $\bar{V}$ is locally Lipschitz around $\bar{x}$, and for $\bar{y}\in G(\bar{x})$,
\[\partial^{\circ}\bar{V}(\bar{x})\subset \partial^{\circ}_xw(\bar{x},\bar{y}).\]
\end{thm}

In fact, Mordukhovich--Sagara's result is applied to a {\bf stochastic reduced form model}. In such models, $\bar{V}$ is not necessarily measurable unless we do not introduce some strong assumptions.\footnote{See, for example, chapter 9 of Stokey and Lucas \cite{SL}.} Therefore, the boundedness of $w$ on the entire space $X^2$ is necessary for their result. However, in deterministic models, such assumptions are too strong. In fact, the boundedness assumption of $w$ excludes the typical $u$ of the RCK model. Therefore, we relax such assumptions. 

Recall the definition of the regularity. A locally Lipschitz real-valued function $f$ is said to be regular at $x$ iff $f'(x;v)=f^{\circ}(x;v)$ for all $v\in X$. The following is our main result.

\begin{thm}\label{thm3}
Let $X$ be a Banach space, $W\subset X$ be nonempty and open, and $\Gamma:\bar{W}\twoheadrightarrow \bar{W}$ be nonempty-valued. Suppose that $w:T\to \mathbb{R}\cup \{-\infty\}$ is defined and continuous (where, $T=\{(x,y)\in \bar{W}^2|y\in \Gamma(x)\}$), and that $\bar{V}:\bar{W}\to \mathbb{R}\cup \{-\infty,+\infty\}$ is real-valued on $W$. Fix $\bar{x}\in W$, and suppose that the following requirements hold.
\begin{enumerate}[i)]
\item There exists an open neighborhood $W'\subset W$ of $\bar{x}$ such that $\Gamma$ is compact-valued and continuous on $W'$, and if $x,x'\in W'$, then $G(x)\cap \Gamma(x')\neq \emptyset$.

\item $w$ is real-valued and locally Lipschitz around $(x,y)$ for all $x\in W'$ and $y\in \mbox{cl.}G(x)$.

\item For every $\bar{y}\in \mbox{cl.}G(\bar{x})\cap W$, $w$ is regular at $(\bar{x},\bar{y})$.

\item $G(\bar{x})\subset W$ and $G$ is upper hemi-continuous at $\bar{x}$.\footnote{This is an additional requirement in this theorem. Because $w$ is not bounded, the Bellman operator is not a contraction in this theorem, and thus there may be a solution to (\ref{B}) other than $\bar{V}$. Hence, the usual arguments for verifying the continuity of $\bar{V}$ using the fixed point theorem cannot be applied in this theorem, and this additional requirement is needed. Note that, if $\bar{V}$ is continuous on some neighborhood of $G(\bar{x})$, then this requirement automatically holds.}
\end{enumerate}
Then, $\bar{V}$ is Lipschitz around $\bar{x}$, and for some $\bar{y}\in \mbox{cl.}G(\bar{x})$,
\[\partial^{\circ}\bar{V}(\bar{x})\subset \partial^{\circ}_xw(\bar{x},\bar{y}).\]
\end{thm}

\begin{proof}
Define
\[\Pi(x)=\{(x_t)\in \bar{W}^{\mathbb{N}}|x_0=x,\ x_{t+1}\in \Gamma(x_t)\mbox{ for all }t\}.\]
Because $\Gamma$ is nonempty-valued, $\Pi(x)\neq \emptyset$. Let $W'$ satisfy all requirements in i). First, choose any $x\in W'$ and $y\in \Gamma(x)$. Choose $(y_t)\in \Pi(y)$, and let $x_0=x, x_1=y$, and $x_t=y_{t-1}$ for $t\ge 2$. Then, $(x_t)\in \Pi(x)$ and $x_1=y$. Therefore,
\begin{align*}
\bar{V}(x)\ge&~\sum_{t=0}^{\infty}\delta^tw(x_t,x_{t+1})\\
=&~w(x,y)+\delta\sum_{t=0}^{\infty}\delta^tw(y_t,y_{t+1}).
\end{align*}
Taking the supremum of the right-hand side, we have that
\[\bar{V}(x)\ge w(x,y)+\delta \bar{V}(y).\]
Because $\bar{V}(x)\in \mathbb{R}$ by assumption, for any $\varepsilon>0$, there exists $(x_t)\in \Pi(x)$ such that
\[\sum_{t=0}^{\infty}\delta^tw(x_t,x_{t+1})\ge \bar{V}(x)-\varepsilon.\]
Let $x_1=y$ and $y_t=x_{t+1}$. Then, $(y_t)\in \Pi(y)$, and thus,
\[\bar{V}(x)-\varepsilon\le w(x,y)+\delta \sum_{t=0}^{\infty}\delta^tw(y_t,y_{t+1})\le w(x,y)+\delta \bar{V}(y).\]
In conclusion, we obtain that
\[\bar{V}(x)\ge \sup\{w(x,y)+\delta \bar{V}(y)|y\in \Gamma(x)\}\ge \bar{V}(x)-\varepsilon,\]
and hence,
\[\bar{V}(x)=\sup\{w(x,y)+\delta \bar{V}(y)|y\in \Gamma(x)\}.\]
By i), $G(x)$ is nonempty, and if $y^*\in G(x)$, then
\[\bar{V}(x)=w(x,y^*)+\delta \bar{V}(y^*)=\max\{w(x,y)+\delta \bar{V}(y)|y\in \Gamma(x)\},\]
which implies that $\bar{V}$ satisfies (\ref{B}) on $W'$.

Second, because $\Gamma(\bar{x})$ is compact, we have that $\mbox{cl.}G(\bar{x})$ is also compact. For each $y\in \mbox{cl.}G(\bar{x})$, there exists an open neighborhood $U_y$ of $(\bar{x},y)$ and $L_y>0$ such that $w$ is $L_y$-Lipschitz on $U_y$. By the compactness of $\mbox{cl.}G(\bar{x})$, there exists $y_1,...,y_k\in \mbox{cl.}G(\bar{x})$ such that $\cup_{i=1}^kU_{y_i}$ is an open neighborhood of $\{\bar{x}\}\times \mbox{cl.}G(\bar{x})$. To shrink $W'$ if necessary, we can assume that there exists an open neighborhood $U$ of $\mbox{cl.}G(\bar{x})$ such that $W'\times U\subset \cup_{i=1}^kU_{y_i}$. Because $G$ is upper hemi-continuous at $\bar{x}$, we can assume without loss of generality that $G(x)\subset U$ for all $x\in W'$. Define $L=\max_iL_{y_i}$.

Third, let $x,x'\in W'$, and choose $y\in G(x)\cap \Gamma(x')$. Then,
\[\bar{V}(x)=w(x,y)+\delta \bar{V}(y),\]
\[\bar{V}(x')\ge w(x',y)+\delta \bar{V}(y).\]
Hence,
\[\bar{V}(x)-\bar{V}(x')\le w(x,y)-w(x',y)\le L\|x-x'\|.\]
By the symmetrical argument, we have that
\[\bar{V}(x')-\bar{V}(x)\le L\|x-x'\|.\]
and thus,
\[|\bar{V}(x)-\bar{V}(x')|\le L\|x-x'\|.\]
Therefore, $\bar{V}$ is $L$-Lipschitz on $W'$.

Choose any $v\in X$, and suppose that $(x_n)$ is a sequence on $W'$ that converges to $\bar{x}$, $(t_n)$ is a positive monotone sequence that converges to $0$, and
\[\lim_{n\to \infty}\frac{\bar{V}(x_n+t_nv)-\bar{V}(x_n)}{t_n}=\bar{V}^{\circ}(\bar{x};v).\]
Let $z_n=x_n+t_nv$. Because $z_n\to \bar{x}$ as $n\to \infty$, we can assume that $z_n\in W'$ for all $n$. Therefore, there exists $y_n\in G(z_n)\cap \Gamma(x_n)$. Because $G$ is upper hemi-continuous, we can assume without loss of generality that $(y_n)$ converges to $\bar{y}\in \mbox{cl.}G(\bar{x})$. Then,
\[\bar{V}(z_n)=w(z_n,y_n)+\delta\bar{V}(y_n),\]
\[\bar{V}(z_n-t_nv)\ge w(z_n-t_nv,y_n)+\delta \bar{V}(y_n),\]
and thus,
\begin{align*}
\bar{V}^{\circ}(\bar{x};v)=&~\lim_{n\to \infty}\frac{\bar{V}(z_n)-\bar{V}(z_n-t_nv)}{t_n}\\
\le&~\limsup_{n\to \infty}\frac{w(z_n,y_n)-w(z_n-t_nv,y_n)}{t_n}\\
=&~\limsup_{n\to\infty}\frac{w(x_n+t_nv,y_n)-w(x_n,y_n)}{t_n}\\
\le&~w^{\circ}(\bar{x},\bar{y};v,0).
\end{align*}
On the other hand, because $w$ is regular at $(\bar{x},\bar{y})$,
\begin{align*}
w_{\bar{y}}^{\circ}(\bar{x};v)=&~\limsup_{x\to \bar{x},\ t\downarrow 0}\frac{w(x+tv,\bar{y})-w(x,\bar{y})}{t}\le\limsup_{(x,y)\to (\bar{x},\bar{y}),\ t\downarrow 0}\frac{w(x+tv,y)-w(x,y)}{t}\\
=&~w^{\circ}(\bar{x},\bar{y};v,0)=w'(\bar{x},\bar{y};v,0)=\lim_{t\downarrow 0}\frac{w(\bar{x}+tv,\bar{y})-w(\bar{x},\bar{y})}{t}\\
\le&~\limsup_{x\to \bar{x},\ t\downarrow 0}\frac{w(x+tv,\bar{y})-w(x,\bar{y})}{t}=w_{\bar{y}}^{\circ}(\bar{x};v).
\end{align*}
Therefore, $w^{\circ}(\bar{x},\bar{y};v,0)=w_{\bar{y}}^{\circ}(\bar{x};v)$, and if $p\in \partial^{\circ}\bar{V}(\bar{x})$, then
\[\langle p,v\rangle\le w_{\bar{y}}^{\circ}(\bar{x};v),\]
which implies that
\[p\in \partial^{\circ}_xw(\bar{x},\bar{y}).\]
This completes the proof. $\blacksquare$
\end{proof}

\begin{table}[t]
\begin{tabular}{|l|c|c|c|c|c|} \hline
& Domain & Concavity & $C^1$ & Boundedness & RCK \\ \hline 
Theorem 1(B-S) & Convex domain & yes & yes & no & CRRA \\ \hline
Theorem 2(M-S) & Banach space & no & no & yes & no \\ \hline
Theorem 3(our) & Open domain & no & no & no & CRRA+$\alpha$ \\ \hline
\end{tabular}
\caption{Comparison with Theorems}
\end{table}

Table 1 compares the main assumptions of Theorems 1-3. Although it does not compare every minor assumption, it should be useful for a rough comparison. The `RCK' column indicates which types of RCK models can be handled. In RCK, neither the domain of $u$ nor the domain of $f$ extends to the entire real line. Therefore, the domain requirement in Theorem 2 is highly incompatible with RCK. Even if this could be remedied, CRRA remains incompatible because this function is unbounded under any $\theta>0$. Meanwhile, the assumptions in Theorem 1 are compatible with CRRA. However, Theorem 3 can handle a wider range of economic situations. Later, we will explain the difference between these two theorems in subsection 3.3.

The following additional result is also useful.

\begin{cor}\label{cor3}
In addition to the assumptions of Theorem \ref{thm3}, suppose that $w_{\bar{y}}$ is strongly differentiable at $\bar{x}$, where $\bar{y}$ is given in Theorem \ref{thm3}. Then, $\bar{V}$ is also strongly differentiable and Hadamard differentiable at $\bar{x}$. Moreover, if $G(\bar{x})\subset \Gamma(x)$ for all $x\in W'$, then for any $y\in G(\bar{x})\cap W$, 
\[D_H\bar{V}(\bar{x})\in \partial^{\circ}_xw(\bar{x},y).\]
\end{cor}

\begin{proof}
By Proposition \ref{prop5}, $\partial^{\circ}_xw(\bar{x},\bar{y})$ is a singleton $\{D_sw_{\bar{y}}(\bar{x})\}$. Therefore,
\[\partial^{\circ}\bar{V}(\bar{x})=\{D_sw_{\bar{y}}(\bar{x})\}.\]
This implies that $\bar{V}$ is both strong differentiable and Hadamard differentiable at $\bar{x}$. Next, suppose that $G(\bar{x})\subset \Gamma(x)$ for all $x\in W'$. Choose any $y\in G(\bar{x})$ and $v\in X$. Let $(t_n)$ be a positive monotone sequence that converges to $0$. We can assume without loss of generality that $\bar{x}+t_nv\in W'$ for all $n$. Thus,
\[\bar{V}(\bar{x})=w(\bar{x},y)+\delta \bar{V}(y),\]
\[\bar{V}(\bar{x}+t_nv)\ge w(\bar{x}+t_nv,y)+\delta \bar{V}(y).\]
Hence,
\begin{align*}
\langle D_H\bar{V}(\bar{x}),v\rangle=&~-\lim_{n\to \infty}\frac{\bar{V}(\bar{x})-\bar{V}(\bar{x}+t_nv)}{t_n}\\
\le&~-\limsup_{n\to \infty}\frac{w(\bar{x},y)-w(\bar{x}+t_nv,y)}{t_n}\\
=&~\liminf_{n\to \infty}\frac{w(\bar{x}+t_nv,y)-w(\bar{x},y)}{t_n}\\
=&~w'(\bar{x},y;v,0)=w^{\circ}(\bar{x},y;v,0).
\end{align*}
By the same arguments as in the proof of Theorem \ref{thm3}, we have that
\[w^{\circ}(\bar{x},y;v,0)=w_y^{\circ}(\bar{x};v).\]
Therefore,
\[D_H\bar{V}(\bar{x})\in \partial^{\circ}_xw(\bar{x},y),\]
as desired. This completes the proof. $\blacksquare$
\end{proof}

Property i) of Theorem \ref{thm3} may be a little difficult to understand. Hence, we provide a sufficient condition for ensuring this property.

\begin{lem}\label{lem1}
Let $X$ be a topological space, $W\subset X$ be nonempty and open, $\Gamma:\bar{W}\twoheadrightarrow \bar{W}$ be nonempty-valued and compact-valued on $W$, $T=\{(x,y)|x\in \bar{W},y\in \Gamma(x)\}$, $w:T\to \mathbb{R}\cup\{-\infty\}$ be continuous, and $\bar{V}:\bar{W}\to \mathbb{R}\cup\{-\infty,+\infty\}$ be a continuous solution to (\ref{B}) on $\bar{W}$ and real-valued on $W$. Moreover, suppose that there exists a continuous real-valued function $H:\bar{W}^2\to \mathbb{R}$ such that if $H(x,y)\le 0$, then $y\in \Gamma(x)$. Fix $\bar{x}\in W$, and suppose that $H(\bar{x},\bar{y})<0$ for all $\bar{y}\in G(\bar{x})$. Then, there exists an open neighborhood $W'\subset W$ of $\bar{x}$ such that if $x,x'\in W'$, then $G(x)\neq \emptyset$ and $G(x)\subset \Gamma(x')$.\footnote{If $X$ is an infinite-dimensional Banach space, then any compact subset of this space has an empty interior. On the other hand, under this theorem, $\Gamma(\bar{x})$ is a compact set and $G(\bar{x})$ is included in its interior. Therefore, for the norm topology, this result can only be applied to finite dimensional cases. However, because this result only requires that $X$ has a topology, this result can also be applicable for an infinite dimensional Banach space with a weaker topology than the norm topology.}
\end{lem}

\begin{proof}
First, because $\bar{V}$ and $w$ are continuous, by Berge's maximum theorem, $G$ is nonempty-valued, compact-valued, and upper hemi-continuous on $W$. Because $G(\bar{x})$ is compact, there exists $\varepsilon>0$ such that $\max_{y\in G(\bar{x})}H(\bar{x},y)<-\varepsilon$. Because $\{\bar{x}\}\times G(\bar{x})$ is compact in $X^2$, there exist open sets $W_1,U$ such that $\{\bar{x}\}\times G(\bar{x})\subset W_1\times U$ and $H(x,y)<0$ for all $x\in W_1, y\in U$. Because $G$ is upper hemi-continuous, there exists an open neighborhood $W_2$ of $\bar{x}$ such that if $x\in W_2$, then $G(x)\subset U$. Let $W'=W_1\cap W_2$. Then, $W'$ is an open neighborhood of $\bar{x}$. Suppose that $x,x'\in W'$. Because $G$ is nonempty-valued, $G(x)\neq \emptyset$. Because $x\in W_2$, $G(x)\subset U$, and because $x'\in W_1$, $H(x',y)<0$ for all $y\in U$, which implies that $U\subset \Gamma(x')$. Therefore, $G(x)\subset \Gamma(x')$, as desired. This completes the proof. $\blacksquare$
\end{proof}

To understand this lemma, consider a model in which $W=\mathbb{R}_+$ and $\Gamma=[0,F(x)]$ for some continuous and increasing function $F$ such that $F(0)=0$. Suppose that for each $x$ sufficiently near to $\bar{x}$, $G$ is single-valued and continuous at $x$, and $G(x)<F(x)$ for all $y\in G(x)$. Define $H(x,y)=y-F(x)$. Then, all the requirements in Lemma \ref{lem1} hold. In this case, if $\varepsilon>0$ is sufficiently small, then for all $x,x'\in W'\equiv ]\bar{x}-\varepsilon,\bar{x}+\varepsilon[$, $G(x)<F(x')$, and thus $G(x)\subset \Gamma(x')$. In particular, the RCK model with a non-boundary optimal path must satisfy this requirement, and thus i) of Theorem \ref{thm3} is naturally met.

\subsection{Application to Economic Models}
In many economic models, Theorem \ref{thm1} is applicable. Theorem \ref{thm3} can be seen as a variant of Theorem \ref{thm1}. We now investigate the relationship between these theorems.

First, consider the RCK model (\ref{RCK}). Recall that $W=\mathbb{R}_{++}\equiv \{k\in \mathbb{R}|k>0\}$ and $\bar{W}=\mathbb{R}_+\equiv \{k\in \mathbb{R}|k\ge 0\}$. To simplify the expression, we define $F(k)=f(k)+(1-d)k$. As we have seen, this model can be transformed into the following reduced form model:
\begin{align}
\max~~~~~&~\sum_{t=0}^{\infty}\delta^tu(F(k_t)-k_{t+1}),\nonumber \\
\mbox{subject to. }&~0\le k_{t+1}\le F(k_t)\mbox{ for all }t\ge 0,\label{RCK2}\\
&~k_0=\bar{k}>0.\nonumber
\end{align}
The following result is fundamental.

\begin{prop}\label{prop7}
Suppose that $u:\mathbb{R}_+\to \mathbb{R}\cup \{-\infty\}$ is continuous and increasing on $\mathbb{R}_+$.\footnote{By this assumption, we have that $u(c)>-\infty$ whenever $c>0$.} Suppose also that $F:\mathbb{R}_+\to \mathbb{R}_+$ is continuous and increasing, and satisfies $F(0)=0$ and $\limsup_{k\to +\infty}\frac{F(k)}{k}<1$. Then, there exists a solution $(k_t^*)$ to (\ref{RCK2}) and $\bar{V}(\bar{k})<+\infty$, and the value function $\bar{V}$ is continuous on $\mathbb{R}_+$. Furthermore, if we assume that there exists $\varepsilon^*>0$ such that $F(k)>k$ whenever $0<k\le \varepsilon^*$, then $\bar{V}(\bar{k})\in \mathbb{R}$ for all $\bar{k}>0$.
\end{prop}

To ensure readability, we put the proof of this proposition into the appendix. Here, we briefly outline the idea of the proof. First, under the assumption $\limsup_{k\to +\infty}\frac{F(k)}{k}<1$, if $K^*>0$ is sufficiently large, the set of sequences that satisfies the constraints in (\ref{RCK2}) must belong to $[0,K^*]^{\mathbb{N}}$. By Tychonoff's theorem, this set is compact with respect to the product topology. Therefore, if the objective function is continuous with respect to this topology, then the existence of a solution follows, and $\bar{V}(\bar{k})<+\infty$. By a standard argument using Berge's theorem, we can show the continuity of $\bar{V}$. If $F(k)>k$ for sufficiently small $k>0$, then a sequence $(k_t)$ such that $k_t\equiv k$ for $t\ge 1$ is feasible in (\ref{RCK2}). Thus, the value of the objective function is real-valued, and thus $\bar{V}(\bar{k})>-\infty$, as desired.

Hence, under very weak assumptions, (\ref{RCK2}) has a solution, $\bar{V}$ is continuous, and $\bar{V}(\bar{k})\in \mathbb{R}$ for all $\bar{k}>0$. In this case, $G$ is nonempty-valued, compact-valued, and upper hemi-continuous.

Now, suppose that $u$ is continuous, increasing, and strictly concave on $\bar{W}$, $C^1$ on $W$, and $\lim_{c\to 0}u'(c)=+\infty$. Moreover, suppose that $f$ is continuous, increasing, and concave on $\bar{W}$, $C^1$ on $W$, and satisfies 1) $f(0)=0$, 2) $\limsup_{k\to \infty}\frac{f(k)}{k}<d$, and 3) there exists $\varepsilon^*>0$ such that if $0<k\le \varepsilon^*$, then $f(k)>dk$. Then, $F(k)=f(k)+(1-d)k$ satisfies that 1) $F(0)=0$, 2) $\limsup_{k\to \infty}\frac{F(k)}{k}<1$, and 3) there exists $\varepsilon^*>0$ such that if $0<k\le \varepsilon^*$, then $F(k)>k$. By Proposition \ref{prop7}, there exists a solution to (\ref{RCK2}), and $\bar{V}$ is finite and continuous on $W$. Moreover,
\[\Gamma(k)=[0,F(k)],\ T=\{(k,k')|k\ge 0,\ k'\ge 0,\ F(k)\ge k'\},\]
and thus $T$ is convex and its interior is nonempty. If $(k,k')\in \mbox{int.}T$, then $k'>0$ and $F(k)-k'>0$. Therefore, $w(k,k')=u(F(k)-k')$ is continuous and concave on $T$, and $C^1$ on $\mbox{int.}T$. Next, choose any $\bar{k}>0$. Then, there exists $\varepsilon>0$ such that $\delta u'(F(\varepsilon))>u'(F(\bar{k})-\varepsilon)$. Note that, because $F(0)=0$, $k_t\equiv 0$ is optimal if $k_0=0$. Hence, if $u(0)=-\infty$, then $\bar{V}(0)=-\infty<\bar{V}(\varepsilon)$, and thus
\[u(F(\bar{k}))+\delta\bar{V}(0)=-\infty<u(F(\bar{k})-\varepsilon)+\delta \bar{V}(\varepsilon),\]
which implies that $0\notin G(\bar{k})$. If $u(0)>-\infty$, then
\begin{align*}
u(F(\bar{k}))+\delta \bar{V}(0)=&~u(F(\bar{k}))+\sum_{t=1}^{\infty}\delta^tu(0)\\
<&~u(F(\bar{k})-\varepsilon)+\delta u(F(\varepsilon))+\sum_{t=2}^{\infty}\delta^tu(0)\\
\le& u(F(\bar{k})-\varepsilon)+\delta \bar{V}(\varepsilon),
\end{align*}
which implies that $0\notin G(\bar{k})$. Next, we can easily check that $\bar{V}$ is concave, and thus there exists $L>0$ such that $\bar{V}$ is $L$-Lipschitz on $]F(\bar{k})-\varepsilon,F(\bar{k})+\varepsilon[$. Choose $\varepsilon'>0$ so small that $\varepsilon'<\varepsilon$ and $u'(\varepsilon')>L$. Then,
\[u(0)+\delta \bar{V}(F(\bar{k}))<u(\varepsilon')+\delta \bar{V}(F(\bar{k})-\varepsilon'),\]
which implies that $F(\bar{k})\notin G(\bar{k})$. Therefore, this model satisfies all requirements of Theorem \ref{thm1}.\footnote{In economics, it is usual that $u$ is assumed to be $C^2$ and $f$ is assumed to be $C^1$. Hence, the assumptions stated in this paragraph are weaker than the usual economic assumptions.}

On the other hand, let $H(\bar{k},k)=-k(F(\bar{k})-k)$. Then, $H(\bar{k},k)\le 0$ implies that $k\in \Gamma(\bar{k})$, and $H(\bar{k},k)<0$ for all $k\in G(\bar{k})$. Hence, Lemma \ref{lem1} is applicable and requirement i) of Theorem \ref{thm3} holds. By proposition \ref{prop6}, $w(k,k')=u(F(k)-k)$ is regular at $(\bar{k},k)$ for all $k\in G(\bar{k})$, and thus requirement iii) holds. Requirements ii) and iv) can be easily verified, and thus this model also satisfies all requirements of Theorem \ref{thm3}. That is, the usual model of the form (\ref{RCK2}) admits all requirements of both theorems.

Next, we drop the $C^1$ assumption for $u$ and $F$. In this case, $u$ and $F$ are still concave, and any concave function on an open and convex set of $\mathbb{R}^n$ is locally Lipschitz. Therefore, $w(k,k')=u(F(k)-k')$ is continuous and concave on $T$, and locally Lipschitz on $\mbox{int.}T$. It is easy to show that $\bar{V}$ is concave, and is thus locally Lipschitz. By Berge's maximum theorem, we have that $G(k)$ is nonempty-valued and upper semi-continuous. If, in addition, we assume that $\lim_{c\to 0}\min \partial u(c)=+\infty$, then by almost the same arguments as in the previous paragraph, we can show that $0,F(\bar{k})\notin G(\bar{k})$. Define $H(\bar{k},k)=-k(F(\bar{k})-k)$. Then, $H(\bar{k},k)\le 0$ implies that $k\in \Gamma(\bar{k})$, and $H(\bar{k},k)<0$ for any $k\in G(\bar{k})$. By Lemma \ref{lem1}, requirement i) of Theorem \ref{thm3} holds. It is easy to check that requirements ii) and iv) are satisfied, and because of Proposition \ref{propx}, iii) holds. Hence, Theorem \ref{thm3} is applicable.

Third, we consider the following alternative situation. Suppose that there is the government, and her policy is represented by some {\bf government expenditure rule} $g(k)$. Because some of $y_t$ is used for $g(k_t)$, the difference equation of $k_t$ has changed: the new equation is
\[k_{t+1}=f(k_t)+(1-d)k_t-g(k_t)-c_t.\]
Hence, we must define $F(k)=f(k)+(1-d)k-g(k)$, and if $g$ is not convex, then $F$ may not be concave, even if $g$ is smooth. Even in this case, the arguments in the previous paragraph can be applicable and we can easily check that i) and iv) of Theorem \ref{thm3} hold. If $u$ is concave, then it is locally Lipschitz. Hence, if $F$ is locally Lipschitz, then $w(k,k')=u(F(k)-k')$ is also locally Lipschitz and ii) holds. Thus, again if $w$ is regular at $(\bar{k},k)$ for all $k\in G(\bar{k})$, then Theorem \ref{thm3} is applicable. On the other hand, Theorem \ref{thm1} is inapplicable even when $u$ and $F$ are $C^1$ because $T$ may not be convex.

In conclusion, Theorem \ref{thm3} is applicable for many models that represent natural economic situations, and in some of these models, Theorem \ref{thm1} is inapplicable.

\section{Conclusion}
In this paper, we examined the envelope theorem for dynamic programming models, and derived a relaxed envelope theorem for these models. Although this result is related to Mordukhovich and Sagara \cite{MS}, the assumption is weakened. We think that this result will be useful in many areas of economic research.

\appendix
\section{Proofs of Several Results}

\subsection{Proof of Proposition \ref{prop1}}
First, we show 1). Suppose that $f:U\to Y$ is locally Lipschitz, $C\subset U$ is compact, and $f$ is not Lipschitz on $C$. Then, for each positive integer $n$, there exists $x_n,y_n\in C$ such that
\[\|f(x_n)-f(y_n)\|>n\|x_n-y_n\|.\]
Because $C$ is compact and $x_n\in C$ for all $n$, we can assume without loss of generality that there exists $x^*=\lim_{n\to \infty}x_n$. By the triangle inequality,
\begin{align*}
n\|x_n-y_n\|<&~\|f(x_n)-f(y_n)\|\le \|f(x_n)-f(x^*)\|+\|f(y_n)-f(x^*)\|\\
\le&~2\max\{\|f(z)-f(x^*)\||z\in C\}\equiv A<+\infty,
\end{align*}
which implies that
\[\|y_n-x^*\|\le \|y_n-x_n\|+\|x_n-x^*\|\le \frac{A}{n}+\|x_n-x^*\|\to 0.\]
Therefore, we have that $\lim_{n\to \infty}y_n=x^*$. Because $f$ is locally Lipschitz, there exist $r>0$ and $L>0$ such that $f$ is $L$-Lipschitz on $B_r(x^*)\cap U$. Because $x_n,y_n\in U$ and $x_n,y_n$ converge to $x^*$, there exists $n$ such that $n>L$ and $x_n,y_n\in B_r(x^*)$. However, this implies that
\[n\|x_n-y_n\|<\|f(x_n)-f(y_n)\|\le L\|x_n-y_n\|,\]
which is a contradiction. 

Next, we show 2). By 1), the `only if' part has already been proved. Hence, it suffices to show the `if' part. Suppose that $f$ is Lipschitz on any compact subset $C\subset U$. Choose any $x\in U$. Because $U$ is an open set, there exists $r>0$ such that $B_{2r}(x)\subset U$. Because $X=\mathbb{R}^n$, $\bar{B}_r\subset B_{2r}(x)$ is a compact set, and thus, $f$ is Lipschitz on this set. Hence, there exists $L>0$ such that if $y,z\in B_r(x)\subset \bar{B}_r(x)$, then $\|f(y)-f(z)\|\le L\|y-z\|$, as desired. This completes the proof. $\blacksquare$

\subsection{Proof of Proposition \ref{prop2}}
Suppose that $(x_n)$ is a sequence on $U$ that converges to $x$, and $(t_n)$ is a positive monotone sequence that converges to $0$. Choose any $r>0$ and $L>0$ such that $f$ is $L$-Lipschitz on $B_r(x)$. Then, there exists $N$ such that if $n\ge N$, then $x_n,x_n+t_nv\in B_r(x)$, and thus,
\[\frac{|f(x_n+t_nv)-f(x_n)|}{t_n}\le \frac{L\|x_n+t_nv-x_n\|}{t_n}=L\|v\|,\]
which implies that $|f^{\circ}(x;v)|\le L\|v\|$.

Next, choose any $v\in X$. By the definition of $\limsup$, there exist a sequence $(x_n)$ that converges to $x$ and a positive monotone sequence $(t_n)$ that converges to $0$ such that
\[\lim_{n\to \infty}\frac{f(x_n+t_nv)-f(x_n)}{t_n}=f^{\circ}(x;v).\]
If $a>0$, then for $s_n=a^{-1}t_n$,
\[f^{\circ}(x;av)\ge \limsup_{n\to \infty}\frac{f(x_n+s_nav)-f(x_n)}{s_n}=a\lim_{n\to \infty}\frac{f(x_n+t_nv)-f(x)}{t_n}=af^{\circ}(x;v).\]
By symmetrical arguments, we can show that $af^{\circ}(x;v)\ge f^{\circ}(x;av)$, and thus, we find that $f^{\circ}(x;av)=af^{\circ}(x;v)$. Because it is clear that $f^{\circ}(x;0)=0$, we have shown that
\[f^{\circ}(x;av)=af^{\circ}(x;v)\]
for any $a\ge 0$.

Choose any $v,w\in X$. Then, there exist a sequence $(x_n)$ that converges to $x$ and a positive monotone sequence $(t_n)$ that converges to $0$ such that
\[f^{\circ}(x;v+w)=\lim_{n\to \infty}\frac{f(x_n+t_n(v+w))-f(x_n)}{t_n}.\]
Define $y_n=x_n+t_nw$. Then, $(y_n)$ converges to $x$, and thus,
\begin{align*}
f^{\circ}(x;v+w)=&~\lim_{n\to \infty}\frac{f(x_n+t_n(v+w))-f(x_n)}{t_n}\\
=&~\lim_{n\to \infty}\left[\frac{f(y_n+t_nv)-f(y_n)}{t_n}+\frac{f(x_n+t_nw)-f(x_n)}{t_n}\right]\\
\le&~\limsup_{n\to \infty}\frac{f(y_n+t_nv)-f(y_n)}{t_n}+\limsup_{n\to \infty}\frac{f(x_n+t_nw)-f(x_n)}{t_n}\\
\le&~f^{\circ}(x;v)+f^{\circ}(x;w).
\end{align*}
This completes the proof. $\blacksquare$

\subsection{Proof of Proposition \ref{prop3}}
Choose any $v\in X$. Suppose that $(x_n)$ is a sequence on $U$ that converges to $x$, $(t_n)$ is a positive monotone sequence that converges to $0$, and
\[\lim_{n\to \infty}\frac{f(x_n+t_nv)-f(x_n)}{t_n}=f^{\circ}(x;v).\]
Then, for $y_n=x_n+t_nv$, $(y_n)$ converges to $x$, and
\[f^{\circ}(x;v)=\lim_{n\to \infty}\frac{(-f)(y_n-t_nv)-(-f)(y_n)}{t_n}\le (-f)^{\circ}(x;-v).\]
Because $f$ and $-f$ can be interchanged, we can easily show that $(-f)^{\circ}(x;-v)\le f^{\circ}(x;v)$. Therefore, we obtain the following equation.
\[(-f)^{\circ}(x;-v)=f^{\circ}(x;v).\]
Hence, if $p\in \partial^{\circ}f(x)$, then
\[(-f)^{\circ}(x;v)=f^{\circ}(x;-v)\ge \langle p,-v\rangle=\langle -p,v\rangle,\]
and thus, $-p\in \partial^{\circ}(-f)(x)$. Conversely, if $-p\in \partial^{\circ}(-f)(x)$, then
\[f^{\circ}(x;v)=(-f)^{\circ}(x;-v)\ge \langle -p,-v\rangle=\langle p,v\rangle,\]
and thus, $p\in \partial^{\circ}f(x)$. Therefore, our claim is correct. This completes the proof. $\blacksquare$

\subsection{Proof of Proposition \ref{prop4}}
Suppose that $p\in \partial f(x)$. By definition,
\begin{align*}
f^{\circ}(x;v)=&~\limsup_{y\to x,\ t\downarrow 0}\frac{f(y+tv)-f(y)}{t}\\
\ge&~\limsup_{t\downarrow 0}\frac{f(x+tv)-f(x)}{t}\ge \lim_{t\downarrow 0}\frac{\langle p,tv\rangle}{t}=\langle p,v\rangle,
\end{align*}
and thus $p\in \partial^{\circ}f(x)$.

Next, suppose that $p\in X'$ and $p\notin \partial f(x)$. Then, there exists $v\in X$ such that $f(x+v)-f(x)<\langle p,v\rangle$, and thus, there exist $a\in \mathbb{R}$ and an open neighborhood $V$ of $x$ such that $y\in V$ implies that $f(y+v)-f(y)<a<\langle p,v\rangle$. If $0<t\le 1$, then $y+tv=(1-t)y+t(y+v)$, and thus
\[\frac{f(y+tv)-f(y)}{t}\le f(y+v)-f(y)<a,\]
which implies that $f^{\circ}(x;v)\le a<\langle p,v\rangle$. Hence, $p\notin \partial^{\circ}f(x)$, which completes the proof. $\blacksquare$

\subsection{Proof of Proposition \ref{prop5}}
First, suppose that $f$ is G\^ateaux differentiable at $x$. Then, for any $v\in X$,
\begin{align*}
\langle D_Gf(x),v\rangle=&~\lim_{t\downarrow 0}\frac{f(x+tv)-f(x)}{t}\\
\le&~\limsup_{y\to x,\ t\downarrow 0}\frac{f(y+tv)-f(y)}{t}=f^{\circ}(x;v),
\end{align*}
which implies that $D_Gf(x)\in \partial^{\circ}f(x)$.

Next, suppose that $f$ is strongly differentiable at $x$. Then, by definition,
\[f^{\circ}(x;v)=\langle D_sf(x),v\rangle\]
for any $v\in X$. Therefore, $D_sf(x)\in \partial^{\circ}f(x)$. Suppose that $p\in X'$ and $p\neq D_sf(x)$. Then, there exists $v\in X$ such that $\langle p,v\rangle>\langle D_sf(x),v\rangle$, which implies that $\langle p,v\rangle>f^{\circ}(x;v)$ and thus $p\notin \partial^{\circ}f(x)$. Hence, $\partial^{\circ}f(x)=\{D_sf(x)\}$.

Conversely, suppose that $\partial^{\circ}f(x)=\{p\}$. First, choose any $v\in X$, and define $V=\{av|a\in\mathbb{R}\}$. For $av\in V$, define $\langle q,av\rangle=af^{\circ}(x;v)$. If $a\ge 0$, then by Proposition \ref{prop2}, $\langle q,av\rangle=f^{\circ}(x;av)$. If $a<0$, then again by Proposition \ref{prop2},
\[0=f^{\circ}(x;0)\le f^{\circ}(x;av)+f^{\circ}(x;-av),\]
and thus,
\[\langle q,av\rangle=-\langle q,-av\rangle=-f^{\circ}(x;-av)\le f^{\circ}(x;av).\]
Therefore, by the Hahn-Banach theorem, there exists $r\in \partial^{\circ}f(x)$ such that the restriction of $r$ to $V$ coincides with $q$. However, because $\partial^{\circ}f(x)=\{p\}$, $r=p$, and in particular,
\[f^{\circ}(x;v)=\langle p,v\rangle.\]
Hence,
\begin{align*}
\liminf_{y\to x,t\downarrow 0}\frac{f(y+tv)-f(y)}{t}=&~-\limsup_{y\to x,t\downarrow 0}\frac{f(y+tv-tv)-f(y+tv)}{t}\\
=&~-f^{\circ}(x;-v)=-\langle p,-v\rangle=\langle p,v\rangle\\
=&~f^{\circ}(x,v)=\limsup_{y\to x,t\downarrow 0}\frac{f(y+tv)-f(y)}{t},
\end{align*}
which implies that
\[\lim_{y\to x,t\downarrow 0}\frac{f(y+tv)-f(y)}{t}=\langle p,v\rangle.\]
Hence, $f$ is strongly differentiable at $x$ and $D_sf(x)=p$.

Finally, suppose that $f$ is strongly differentiable at $x$, and let $p=D_sf(x)$. Because $f$ is locally Lipschitz, there exist $r>0$ and $L>0$ such that $f$ is $L$-Lipschitz on $B_r(x)$. Choose any compact set $C\subset X$. Fix $\varepsilon>0$, and choose any $v\in C$. Then, there exists $\delta(v)\in ]0,r[$ such that if $\|y-x\|<\delta(v)$ and $0<t<\delta(v)$, then
\[|f(y+tv)-f(y)-t\langle p,v\rangle|<t\frac{\varepsilon}{L+\|p\|+1}.\]
Let
\[r'=\min\left\{\frac{\varepsilon}{L+\|p\|+1},1\right\}.\]
If $w\in C$ and $\|w-v\|<r'$, then
\begin{align*}
|f(x+tw)-f(x)-t\langle p,w\rangle|\le&|f((x+t(w-v))+tv)-f(x+t(w-v))-t\langle p,v\rangle|\\
&~+|f(x+t(w-v))-f(x)|+t|\langle p,w-v\rangle|\\
<&~t\frac{\varepsilon}{L+\|p\|+1}+t(L+\|p\|)\|w-v\|<t\varepsilon.
\end{align*}
Because $\{B_{r'}(v)\}_{v\in C}$ is an open covering of $C$, there is a finite subcovering $\{B_{r'}(v_i)\}_{i=1}^N$. Let $\delta=\min\{\delta(v_1),...,\delta(v_N)\}$. If $0<t<\delta$, then for any $v\in C$,
\[|f(x+tv)-f(x)-t\langle p,v\rangle|<t\varepsilon.\]
This implies that $p=D_Hf(x)$. Clearly, $D_Hf(x)=D_Gf(x)$, and
\[f^{\circ}(x;v)=\langle D_sf(x),v\rangle=\langle D_Gf(x),v\rangle=f'(x;v),\]
which implies that $f$ is regular at $x$. This completes the proof. $\blacksquare$

\subsection{Proof of Proposition \ref{prop6}}
First, choose any bounded set $B\subset X$, and choose $M>0$ such that $\|v\|\le M$ for all $v\in B$. Fix $\varepsilon>0$, and choose $r>0$ such that if $y\in B_r(x)$, then
\[\|D_Gf(y)-D_Gf(x)\|<\frac{\varepsilon}{M}.\]
Let $t>0$ be so small that $t<\frac{r}{M}$. Then, $x+tv\in B_r(x)$. Because of the mean value theorem, there exists $\theta\in [0,1]$ such that for all $t>0$,
\[\frac{f(x+tv)-f(x)}{t}=\langle D_Gf(x+\theta tv),v\rangle.\]
Therefore,
\begin{align*}
&~\left|\frac{f(x+tv)-f(x)}{t}-\langle D_Gf(x),v\rangle\right|\\
=&~|\langle D_Gf(x+\theta tv)-D_Gf(x),v\rangle|\\
\le&~\|D_Gf(x+\theta tv)-D_Gf(x)\|\|v\|<\varepsilon,
\end{align*}
as desired. Hence, $D_Gf(x)=Df(x)$. Next, choose any $v\in X$. Let $(x_n)$ be a sequence on $U$ that converges to $x$, and $(t_n)$ be a decreasing positive sequence that converges to $0$. Without loss of generality, we can assume that $x_n,x_n+t_nv\in B_r(x)$ for all $n$. Then, there exists $\theta_n\in [0,1]$ such that
\[\frac{f(x_n+t_nv)-f(x_n)}{t}=\langle D_Gf(x_n+\theta_nt_nv),v\rangle\to \langle D_Gf(x),v\rangle,\]
which implies that $f$ is strongly differentiable at $x$ and
\[D_sf(x)=D_Gf(x).\]
This completes the proof. $\blacksquare$

\subsection{Proof of Proposition \ref{propx}}
We treat only the case in which $f$ is convex and $L$-Lipschitz: the remaining cases can easily be proved using the result of this case. Choose any $x\in U$ and $v\in X$. Then, for $t,s>0$ with $s>t$,
\[-L\|v\|\le \frac{f(x+tv)-f(x)}{t}\le \frac{f(x+sv)-f(x)}{s}\le L\|v\|.\]
Therefore,
\[f'(x;v)=\inf_{t>0}\frac{f(x+tv)-f(x)}{t}\in [-L\|v\|,L\|v\|].\]
Now, choose a sequence $(x_n)$ on $U$ that converges to $x$ and a sequence $(t_n)$ of positive real numbers that converges to $0$ such that
\[f^{\circ}(x;v)=\lim_{n\to \infty}\frac{f(x_n+t_nv)-f(x_n)}{t_n}.\]
Fix $\delta>0$, and let $s_n=\max\{t_n,\|x_n-x\|/\delta\}$. Then, $s_n\to 0$. Because $t\mapsto \frac{f(x_n+tv)-f(x_n)}{t}$ is nondecreasing,
\begin{align*}
f^{\circ}(x;v)\ge&~\limsup_{n\to \infty}\frac{f(x_n+s_nv)-f(x_n)}{s_n}\\
\ge&~\liminf_{n\to \infty}\frac{f(x_n+s_nv)-f(x_n)}{s_n}\\
\ge&~\lim_{n\to \infty}\frac{f(x_n+t_nv)-f(x_n)}{t_n}=f^{\circ}(x;v),
\end{align*}
and thus, we can assume without loss of generality that $t_n=s_n$ and $\|x_n-x\|\le \delta t_n$. Then,
\[\left|\frac{f(x_n+t_nv)-f(x_n)}{t_n}-\frac{f(x+t_nv)-f(x)}{t_n}\right|\le 2L\delta,\]
and thus,
\[f^{\circ}(x;v)\ge f'(x;v)\ge f^{\circ}(x;v)-2L\delta.\]
Because $\delta>0$ is arbitrary, we obtain that $f^{\circ}(x;v)=f'(x;v)$. This completes the proof. $\blacksquare$

\subsection{Proof of Proposition \ref{prop7}}
Because $\limsup_{k\to +\infty}\frac{F(k)}{k}<1$, there exists $K^*$ such that if $k\ge K^*$, then $F(k)\le k$. Without loss of generality, we can assume that $\bar{k}<K^*$, and thus the feasible sequence $(k_t)$ on the model (\ref{RCK2}) must satisfy $k_t\in [0,K^*]$ for all $t\ge 1$.

Let
\[\Pi(k)=\{(k_t)\in [0,K^*]^{\mathbb{N}}|k_0=k,\ 0\le k_{t+1}\le F(k_t)\},\]
\[\Pi^*=\cup_{k\in [0,K^*]}\Pi(k).\]
Choose any $k\in [0,K^*]$ and $(k_t)\in \Pi(k)$. By the assumption of $K^*$, $F(k_t)-k_{t+1}\le K^*$, and thus, $u(F(k_t)-k_{t+1})\le u(F(K^*))$. This implies that the following function
\[U((k_t))=\sum_{t=0}^{\infty}\delta^tu(F(k_t)-k_{t+1})\]
is well-defined and bounded from above. By Tychonoff's theorem, $[0,K^*]^{\mathbb{N}}$ is compact with respect to the product topology, and because $\Pi(k)$ is a closed subset of this set, it is also compact. We first show that $U$ is actually continuous on $\Pi^*$.

Choose a convergent sequence $(k_t^n)$ on $\Pi^*$ and $\lim_{n\to \infty}k_t^n=k_t^*$ for each $t$. Because $0\le k_t^n\le K^*$ and $0\le k_{t+1}^n\le F(k_t^n)$ for all $n$, we have that $0\le k_t^*\le K^*$ and $0\le k_{t+1}^*\le F(k_t^*)$, and thus $(k_t^*)\in \Pi(k_0^*)\subset \Pi^*$.  First, suppose that $U((k_t^*))=-\infty$. Choose any $a\in \mathbb{R}$. Then, there exists $T>0$ such that
\[\sum_{t=0}^T\delta^t(u(F(k_t^*)-k_{t+1}^*)-u(F(K^*)))+\frac{u(F(K^*))}{1-\delta}<a.\]
This implies that, there exists $N$ such that if $n\ge N$, then
\[\sum_{t=0}^T\delta^t(u(F(k_t^n)-k_{t+1}^n)-u(F(K^*)))+\frac{u(F(K^*))}{1-\delta}<a.\]
Therefore,
\begin{align*}
&~\sum_{t=0}^{\infty}\delta^tu(F(k_t^n)-k_{t+1}^n)\\
\le&~\sum_{t=0}^T\delta^t(u(F(k_t^n)-k_{t+1}^n)-u(F(K^*)))+\sum_{t=0}^{\infty}\delta^tu(F(K^*))<a,
\end{align*}
which implies that
\[\limsup_{n\to \infty}U((k_t^n))=-\infty,\]
and in this case, $\lim_{n\to \infty}U((k_t^n))=U((k_t^*))$. Hence, we can assume that $U((k_t^*))>-\infty$. In particular, $u(F(k_t^*)-k_{t+1}^*)>-\infty$ for all $t\ge 0$.

Next, choose any $\varepsilon>0$. There exists $N_t$ such that if $n\ge N_t$, then
\[|u(F(k_t^n)-k_{t+1}^n)-u(F(k_t^*)-k_{t+1}^*)|<(1-\delta)\varepsilon.\]
Define $M_1=N_1$, and if $M_T$ is defined, then define $M_{T+1}=\max\{N_{T+1},M_T+1\}$. Then, $(M_T)$ is an increasing sequence. If $n\ge M_T$,
\begin{align*}
&~\sum_{t=0}^{\infty}\delta^tu(F(k_t^n)-k_{t+1}^n)\\
=&~\sum_{t=0}^T\delta^tu(F(k_t^n)-k_{t+1}^n)+\sum_{t=T+1}^{\infty}\delta^tu(F(k_t^n)-k_{t+1}^n)\\
<&~\sum_{t=0}^T\delta^tu(F(k_t^*)-k_{t+1}^*)+(1-\delta^{T+1})\varepsilon+\frac{\delta^{T+1}}{1-\delta}u(F(K^*))\\
\to&~U((k_t^*))+\varepsilon,
\end{align*}
and
\begin{align*}
&~\sum_{t=0}^{\infty}\delta^tu(F(k_t^*)-k_{t+1}^*)\\
=&~\sum_{t=0}^T\delta^tu(F(k_t^*)-k_{t+1}^*)+\sum_{t=T+1}^{\infty}\delta^tu(F(k_t^*)-k_{t+1}^*)\\
<&~\sum_{t=0}^T\delta^tu(F(k_t^n)-k_{t+1}^n)+(1-\delta^{T+1})\varepsilon+\frac{\delta^{T+1}}{1-\delta}u(F(K^*))\\
\to&~U((k_t^n))+\varepsilon.
\end{align*}
Therefore, we have that
\[|U((k_t^n))-U((k_t^*))|\le \varepsilon\]
for sufficiently large $n$, as desired. 

Hence, $U$ is continuous on $\Pi^*$. Next, we will show that the set-valued mapping $\Pi$ is continuous. Choose $k\in [0,K^*]$ and let $(k^n)$ be a sequence on $[0,K^*]$ that converges to $k$. First, suppose that there exist an open neighborhood $U$ of $\Pi(k)$ and a sequence $((k_t^n))$ on $\Pi^*$ such that $(k_t^n)\in \Pi(k^n)$ for each $n$, and for all $N$, there exists $n\ge N$ such that $(k_t^n)\notin U$. Taking a subsequence, we can assume that $(k_t^n)\notin U$ for all $n$. Because $(k_t^n)$ is a sequence of $[0,K^*]^{\mathbb{N}}$, we can assume without loss of generality that $\lim_{n\to \infty}k_t^n=k_t^*$ for each $t\ge 0$. Using the same arguments as in the previous paragraphs, we can show that $(k_t^*)\in \Pi(k)\subset U$, which contradicts our initial assumption. Therefore, we have that $\Pi$ is upper hemi-continuous. Second, suppose that $U$ is an open set in $\Pi^*$ such that there exists $(k_t^*)\in U\cap \Pi(k)$. Define $k_0^n=k^n$, and if $k_t^n$ is defined, then define $k_{t+1}^n=\min\{k_{t+1}^*,F(k_t^n)\}$. Because $F$ is continuous, $\lim_{n\to\infty}k_1^n=k_1^*$, and if $\lim_{n\to \infty}k_t^n=k_t^*$, then $\lim_{n\to \infty}k_{t+1}^n=k_{t+1}^*$. Hence, $(k_t^n)\in \Pi(k^n)$ and $\lim_{n\to \infty}(k_t^n)=(k_t^*)$, which implies that $(k_t^n)\in \Pi(k^n)\cap U$ for sufficiently large $n$. Therefore, $\Pi$ is lower hemi-continuous, as desired.

Therefore, $\Pi$ is a compact-valued continuous mapping on $[0,K^*]$, and by Berge's maximum theorem, the solution mapping to (\ref{RCK2}) is nonempty-valued, and the value function $\bar{V}$ is continuous on $[0,K^*]$. Let $(k_t^*)$ be the solution to (\ref{RCK2}) with $k_0^*=\bar{k}$. Because $k_t^*\le K^*$, we have that
\[\bar{V}(\bar{k})=\sum_{t=0}^{\infty}\delta^tu(F(k_t^*)-k_{t+1}^*)\le \frac{1}{1-\delta}u(F(K^*))<+\infty.\]
Suppose that there exists $\varepsilon^*$ in the claim of this proposition. Let $k^+=\min\{\varepsilon^*,\bar{k}\}$. Then,
\[\bar{V}(\bar{k})\ge u(F(\bar{k})-k^+)+\sum_{t=1}^{\infty}\delta^tu(F(k^+)-k^+)>-\infty.\]
Therefore, $\bar{V}(\bar{k})\in \mathbb{R}$. This completes the proof. $\blacksquare$

\end{document}